\documentclass[twoside,12pt]{article}
\usepackage[T1]{fontenc}
\usepackage[utf8]{inputenc}
\usepackage{lmodern}
\usepackage{mathrsfs}
\usepackage{amssymb, amsmath, mathtools, amsthm}
\usepackage{graphicx}
\usepackage{color}
\usepackage{float, caption}
\usepackage{enumitem}
\usepackage{booktabs}
\usepackage{microtype}
\usepackage[top=2cm, bottom=2cm, left=2cm, right=2cm]{geometry}
\usepackage[colorlinks=true,linkcolor=blue,citecolor=blue,urlcolor=blue]{hyperref}

\IfFileExists{epsf.tex}{\input{epsf}}{}

\newcommand{\bd}{\begin{description}}
\newcommand{\ed}{\end{description}}
\newcommand{\bi}{\begin{itemize}}
\newcommand{\ei}{\end{itemize}}
\newcommand{\be}{\begin{enumerate}}
\newcommand{\ee}{\end{enumerate}}
\newcommand{\beq}{\begin{equation}}
\newcommand{\eeq}{\end{equation}}
\newcommand{\beqs}{\begin{eqnarray*}}
\newcommand{\eeqs}{\end{eqnarray*}}

\definecolor{DarkGreen}{rgb}{0.2, 0.6, 0.3}

\numberwithin{equation}{section}

\newtheorem{theorem}{Theorem}[section]

\newtheorem{lemma}[theorem]{Lemma}
\newtheorem{problem}[theorem]{Problem}

\newtheorem{remark}{Remark}[section]
\begin{document}

\title{\textbf{A complete solution to the biased Alon-Krivelevich-Spencer-Szab\'o criterion problem for the discrepancy game\footnote{Supported by the National Science Foundation of China
(Nos. 12471329, 12401461 and 12061059).}}}

\author{Yaping Mao\footnote{Academy of Plateau Science and Sustainability, and School of Mathematics and Statistics, Qinghai Normal University, Xining, Qinghai 810008, China. {\tt yapingmao@outlook.com; myp@qhnu.edu.cn}}, \ \ 
Meiqin Wei\footnote{School of Science, Shanghai Maritime University, Shanghai 201306, China. {\tt weimeiqin8912@163.com}}, \ \ 
Gang Yang\footnote{Corresponding author. Faculty of Environment and Information Sciences, Yokohama National University, 79-2 Tokiwadai, Hodogaya-ku, Yokohama 240-8501, Japan. {\tt gangyang98@outlook.com}}}
\date{}
\maketitle

\begin{abstract}
Let \(H=(V,\mathcal E)\) be a finite hypergraph. For positive integers
\(p\) and \(q\), the \((p:q)\)-biased discrepancy game on \(H\) is
played in complete rounds. In each round, Balancer first claims \(p\)
previously unclaimed vertices, and then Unbalancer claims \(q\)
previously unclaimed vertices. Let \(B\) and \(U\) be the final sets of
vertices claimed by Balancer and Unbalancer, respectively. For an edge
\(e\in\mathcal E\), define
$D_e = q|B\cap e|-p|U\cap e|
     = (p+q)|B\cap e|-p|e|$.
Thus \(D_e\) measures the deviation of Balancer's share of \(e\) from
the density \(p/(p+q)\).
In 2005, Alon, Krivelevich, Spencer and Szab\'o proved a Chernoff-type potential
criterion for the unbiased alternating discrepancy game, corresponding
to the case \(p=q=1\), and asked for a biased analogue for general $p,q$. In this paper, we give an affirmative answer of their problem and prove a
complete biased analogue in the complete-round formulation. More
precisely, for every finite hypergraph \(H=(V,\mathcal E)\) and every
fixed bias \((p:q)\), we give an explicit exponential condition under
which Balancer has a strategy forcing
$-L_e^- \le D_e \le L_e^+$
for every  $e\in\mathcal E$,
where \(L_e^+\) and \(L_e^-\) are prescribed edge-dependent target
values.
\vspace{1mm}

\noindent {\bf Keywords:} Discrepancy game; Biased game; Hypergraph discrepancy; Potential method.

\noindent {\bf AMS subject classification 2020:} 05C57; 05C65; 91A46; 05D40.
\end{abstract}

\section{Introduction}
Discrepancy theory has its roots in the classical study of irregularities of
distribution, initiated by van der Corput and developed substantially by Roth
\cite{vanderCorput1935,Roth1954}.  In the combinatorial setting, one studies a
finite set system, or equivalently a hypergraph \(H=(V,\mathcal E)\), and asks
whether the vertex set can be divided into two parts so that every hyperedge is
nearly balanced.  
This hypergraph viewpoint was explicitly
introduced and developed by Beck in his work on discrepancies of integer
sequences and hypergraphs \cite{Beck1981Roth}; an early foundational result in
this direction is the Beck--Fiala theorem \cite{BeckFiala1981}.  Hypergraph
discrepancy has since become a central topic in combinatorics, with deep
connections to probabilistic methods, set systems, geometric discrepancy and
positional games; see, for example,
\cite{Beck,Chazelle,Matousek,Spencer85}.

A \emph{positional game} is usually described by a pair
\((X,\mathcal F)\), where \(X\) is a finite board and
\(\mathcal F\subseteq 2^X\) is a family of distinguished subsets,
often called \emph{target sets} or \emph{winning sets}. Equivalently, one may view
\((X,\mathcal F)\) as a hypergraph whose vertices are the elements of
\(X\) and whose hyperedges are the members of \(\mathcal F\).
The systematic study of positional games goes back to the work of
Hales and Jewett \cite{HJ63} and Erd\H{o}s and Selfridge
\cite{ErdosSelfridge}; see also the monograph of Hefetz,
Krivelevich, Stojakovi\'c and Szab\'o \cite{HKSS14}.

In the classical weak and strong versions of positional games, the
players alternately claim previously unclaimed elements of the board,
and the objective is to occupy all elements of some winning set.
Discrepancy games follow the same hypergraph and claiming framework,
but the objective is quantitative rather than existential: instead of
trying to complete a winning set, one evaluates the final partition of
the board by the discrepancies it induces on all hyperedges.

In this quantitative direction, Alon, Krivelevich, Spencer and
Szab\'o \cite{AKSS} introduced a game-theoretic form of hypergraph
discrepancy. The game is played on the vertex set of a finite hypergraph
\(H=(V,\mathcal E)\). Here and throughout, to \emph{claim} a vertex means
to choose a previously unclaimed vertex and assign it permanently to that
player. In the fair alternating game, Balancer and Unbalancer claim
previously unclaimed vertices one at a time. At the end of the game,
Balancer's vertices are labelled \(+1\) and Unbalancer's vertices are
labelled \(-1\). Thus each edge \(e\in\mathcal E\) receives an edge-sum,
and Balancer's aim is to keep all these edge-sums small in absolute
value.

Alon, Krivelevich, Spencer and Szab\'o \cite{AKSS} proved a Chernoff-type sufficient condition under which Balancer can achieve this aim.  Their argument is based on a potential function, in the spirit of the Erd\H{o}s--Selfridge criterion for Maker--Breaker games \cite{ErdosSelfridge} and Beck's fake probabilistic method \cite{Beck81,Beck}.  Related potential arguments also appear in biased games and pseudorandom graph games; see, for example, Lu's matching-game criterion \cite{Lu}, the game of JumbleG of Frieze, Krivelevich, Pikhurko and Szab\'o \cite{FKPS}, and the Bart--Moe games studied by Hefetz, Krivelevich and Szab\'o \cite{HKS}.

In the notation of the present paper, the main result of Alon,
Krivelevich, Spencer and Szab\'o may be stated as follows
\cite[Theorem~1]{AKSS}.
The final sets of vertices claimed by Balancer and Unbalancer are denoted by $B$ and $U$, respectively.
\begin{theorem}[Alon,
Krivelevich, Spencer and Szab\'o  \cite{AKSS}]\label{thm:AKSS}
Let $H=(V,\mathcal E)$ be a finite hypergraph with
$\mathcal E=\{e_1,\ldots,e_m\}$.
Suppose that $|V|$ is even, and consider the fair alternating discrepancy
game on $V$ in which Balancer moves first.  Let
$b=(b_1,\ldots,b_m)$
be a vector of nonnegative real numbers.  If
\[
        \sum_{i=1}^m
        \exp\!\left\{-\frac{b_i^2}{2|e_i|}\right\}
        \le \frac12,
\]
then Balancer has a strategy guaranteeing
$\bigl||B\cap e_i|-|U\cap e_i|\bigr|\le b_i$
for every $i=1,\ldots,m$.
In particular, Balancer has such a strategy whenever
$b_i\ge \sqrt{2|e_i|\log(2m)}$
for every $i=1,\ldots,m$.
\end{theorem}

In the notation of the present paper, Beck's uniform biased discrepancy
theorem may be stated as follows \cite[Theorem~17.5]{Beck}.

\begin{theorem}[Beck \cite{Beck}]\label{thm:beck-uniform}
Let $H=(V,\mathcal E)$ be a finite $n$-uniform hypergraph, and put
$m=|\mathcal E|$. In the $(p:q)$-game, Balancer claims $p$ previously
unclaimed vertices and then Unbalancer claims $q$ previously unclaimed
vertices in each complete round.  Suppose that Balancer moves first.
Then Balancer has a strategy such that, at the end of the game, for every
edge $e\in\mathcal E$,
\[
        \frac{p-\alpha}{p+q}\,n
        <
        |B\cap e|
        <
        \frac{p+\alpha}{p+q}\,n,
\]
where
\[
        \alpha
        =
        \left(
        1+
        O\!\left(
        pq\sqrt{\frac{\log m}{(p+q)n}}
        \right)
        \right)
        2pq
        \sqrt{\frac{\log m}{(p+q)n}} .
\]
Equivalently, Balancer can force
\[
        \left|
        |B\cap e|-\frac{p}{p+q}n
        \right|
        <
        \frac{\alpha}{p+q}n
        \qquad\text{for every }e\in\mathcal E .
\]
\end{theorem}

We shall use the following biased version of the discrepancy game.  Let $p,q$ be positive integers.  In the \textit{$(p:q)$-game}, a complete round consists of two parts: first Balancer claims $p$ previously unclaimed vertices, and then Unbalancer claims $q$ previously unclaimed vertices.  In the complete-round formulation considered below, the number of vertices of the board is divisible by $p+q$, so that all vertices are claimed after an integral number of complete rounds.

At the end of their paper, Alon, Krivelevich, Spencer and Szab\'o \cite{AKSS} posed the corresponding biased problem for general hypergraphs and general $p,q$. We record it in the following form.

\begin{problem}[Alon,
Krivelevich, Spencer and Szab\'o  \cite{AKSS}]\label{prob:AKSS-biased}
Let $p,q$ be positive integers.  Establish a discrepancy criterion for the $(p:q)$-game on a finite hypergraph $H=(V,\mathcal E)$.  In particular, give conditions on $H$ which guarantee that Balancer can force his final share of every edge $e\in\mathcal E$ to be close to $p|e|/(p+q)$.
\end{problem}

For an edge $e\in\mathcal E$, the natural target density for Balancer is $p/(p+q)$.  We measure the signed deviation from this target by
\begin{equation}
\label{eq:D}
        D_e=q|B\cap e|-p|U\cap e|
        =(p+q)|B\cap e|-p|e|.
\end{equation}
Thus controlling $D_e$ is exactly equivalent to controlling the deviation of $|B\cap e|$ from its target value $p|e|/(p+q)$.

Problem~\ref{prob:AKSS-biased} is genuinely different from the fair alternating game.  In the fair game, after a move of Balancer, Unbalancer claims only one vertex before Balancer moves again. In the $(p:q)$-game, a complete round may affect the same edge several times. 
We shall therefore regard each complete round as a single step.  Balancer first chooses a set of $p$ unclaimed vertices, and then Unbalancer chooses a set of $q$ vertices from those which remain unclaimed. Balancer's choice will be made as one $p$-set, rather than as $p$ independent single choices.

We now give the notation used in the criterion. Empty edges have $D_e=0$ and will be omitted.  For each edge $e\in\mathcal E$, we prescribe two nonnegative numbers $L_e^+$ and $L_e^-$. The number $L_e^+$ bounds the allowed positive deviation of $D_e$, while $L_e^-$ bounds the allowed negative deviation of $D_e$.  Thus the desired final estimate for $e$ is
$-L_e^-\le D_e\le L_e^+$.
We also associate two positive parameters $t_e$ and $s_e$ with $e$. The parameter $t_e$ is used for the upper inequality $D_e\le L_e^+$, while $s_e$ is used for the lower inequality $D_e\ge -L_e^-$. During the game, $c_e$ denotes the number of vertices of $e$ which have already been claimed.

We next introduce two auxiliary functions. For the upper inequality, the relevant exponential expression is $e^{tD_e}$. If Balancer claims one vertex of $e$, then $D_e$ increases by $q$, and this expression is multiplied by $e^{qt}$. If Unbalancer claims one vertex of $e$, then $D_e$ decreases by $p$, and this expression is multiplied by $e^{-pt}$. We divide both one-vertex factors by the same positive number $R$. The upper factor $R^+_{p,q}(t)$ is defined to be the positive solution of
\begin{equation}
\label{eq:R-plus}
        p\frac{e^{qt}}{R}
        +
        \left(\frac{e^{-pt}}{R}\right)^q
        =
        p+1 .
\end{equation}

For the lower inequality, the relevant exponential expression is $e^{-sD_e}$. If Balancer claims one vertex of $e$, this expression is multiplied by $e^{-qs}$; if Unbalancer claims one vertex of $e$, it is multiplied by $e^{ps}$. The lower factor $R^-_{p,q}(s)$ is defined to be the positive solution of
\begin{equation}
\label{eq:R-minus}
        p\frac{e^{-qs}}{R}
        +
        \left(\frac{e^{ps}}{R}\right)^q
        =
        p+1.
\end{equation}
The left hand sides of \eqref{eq:R-plus} and \eqref{eq:R-minus} are continuous and strictly decreasing functions of $R>0$, tend to infinity as $R\to0^+$, and tend to zero as $R\to\infty$. Hence the positive solutions are well defined and unique.

\begin{theorem}
\label{thm:main}
Let $H=(V,\mathcal E)$ be a finite hypergraph with no empty edges. Suppose that $|V|$ is divisible by $p+q$, and that the $(p:q)$-game on $V$ is played in complete rounds with Balancer moving first in each round. For every edge $e\in\mathcal E$, let $L_e^+,L_e^-\ge0$. Assume that there exist parameters $t_e,s_e>0$ such that
\begin{equation}
\label{eq:main-condition}
        \sum_{e\in\mathcal E} e^{-t_eL_e^+}
        \bigl(R^+_{p,q}(t_e)\bigr)^{|e|}
        +
        \sum_{e\in\mathcal E} e^{-s_eL_e^-}
        \bigl(R^-_{p,q}(s_e)\bigr)^{|e|}
        \le 1.
\end{equation}
Then Balancer has a strategy guaranteeing, at the end of the game,
\[
        -L_e^-\le q|B\cap e|-p|U\cap e|\le L_e^+
        \qquad\text{for all }e\in\mathcal E.
\]
Equivalently,
\[
        -\frac{L_e^-}{p+q}\le
        |B\cap e|-\frac{p}{p+q}|e|\le
        \frac{L_e^+}{p+q}
        \qquad\text{for all }e\in\mathcal E.
\]
\end{theorem}

Theorem~\ref{thm:main} should be compared with Theorems~\ref{thm:AKSS}
and~\ref{thm:beck-uniform} as follows. First, when $p=q=1$, equations
\eqref{eq:R-plus} and \eqref{eq:R-minus} give $R^+_{1,1}(t)=R^-_{1,1}(t)=\cosh t$.
Thus, by taking
        $L_e^+=L_e^-=b_e$,
        and 
        $t_e=s_e=\frac{b_e}{|e|}$,
and using the elementary inequality $\cosh x\le \exp(x^2/2)$, the condition
in Theorem~\ref{thm:main} reduces to
\[
        2\sum_{e\in\mathcal E}
        \exp\!\left\{-\frac{b_e^2}{2|e|}\right\}
        \le 1.
\]
This is precisely the condition in Theorem~\ref{thm:AKSS}. Hence
Theorem~\ref{thm:main} recovers the fair alternating discrepancy criterion
of Alon, Krivelevich, Spencer and Szab\'o.

Second, Beck's Theorem~\ref{thm:beck-uniform} is a biased result, but it is
uniform: all edges have the same size $n$, and a single common error term is
used for every edge.  Theorem~\ref{thm:main} gives an edge-dependent
criterion. The hypergraph need not be uniform, the upper and lower target
values $L_e^+$ and $L_e^-$ may vary with $e$, and the exponential parameters
$t_e$ and $s_e$ may also vary with $e$. In the special case of an
$n$-uniform hypergraph with common targets and common parameters,
Theorem~\ref{thm:main} gives a Beck-type Chernoff-scale estimate around the
density $p/(p+q)$. 
The formal specializations of Theorem~\ref{thm:main} yielding
Theorems~\ref{thm:AKSS} and~\ref{thm:beck-uniform} are recorded after the
proof of Theorem~\ref{thm:main}; see Section~\ref{sec:proof-main}. 

We finally state the pseudo-random graph application. This is the
biased-density version of the graph application suggested by
Problem~\ref{prob:AKSS-biased}.  For disjoint sets $S,T\subseteq[n]$, let
$e_G(S,T)$ denote the number of edges of $G$ with one endpoint in $S$ and the
other in $T$.

\begin{theorem}\label{thm:pseudorandom}
Fix positive integers $p,q$, and put
\[
        \rho=\frac{p}{p+q}.
\]
There exist constants $C=C(p,q)$ and $\eta_0=\eta_0(p,q)>0$ such that the
following holds. Let $0<\eta\le\eta_0$, let $n\ge3$, and suppose that
\[
        \eta^3 n\ge C\log(e/\eta).
\]
In the $(p:q)$-game played on the edge set of $K_n$, Balancer has a strategy
to build a graph $G_B$ satisfying
\[
        \left|d_{G_B}(v)-\rho(n-1)\right|
        \le
        C\sqrt{n\log n}+O_{p,q}(1)
        \qquad\text{for every }v\in[n],
\]
and, for every pair of disjoint sets $S,T\subseteq[n]$ with
$|S|,|T|\ge\eta n$,
\[
        \left|
        \frac{e_{G_B}(S,T)}{|S||T|}-\rho
        \right|
        \le \eta .
\]
\end{theorem}

The rest of the paper is organized as follows. In
Section~\ref{sec:proof-main}, we prove Theorem~\ref{thm:main}; immediately after the proof, we record the formal derivations of
Theorems~\ref{thm:AKSS} and~\ref{thm:beck-uniform} from Theorem~\ref{thm:main}.
In Section~\ref{sec:proof-pseudorandom}, we apply Theorem~\ref{thm:main}
to the edge game on $K_n$ and prove Theorem~\ref{thm:pseudorandom}.

\section{Proof of Theorem~\ref{thm:main}}\label{sec:proof-main}

We now prove two elementary lemmas which will be used to choose Balancer's $p$ vertices in one round. The first lemma is a one-edge estimate. 

Let $C>0$.
For integers $0\le i\le p$ and $0\le j\le q$, define
\begin{equation}\label{hcij}
H_C(i,j) = (C-1)\Bigl((p-i)jC^i-i(q-j)C^{i-1}\Bigr).
\end{equation}

\begin{lemma}\label{lem-1}
Let $C, M>0$.
For all integers $0\le i\le p$ and $0\le j\le q$, if $pC+M^q\le p+1$, then
\[
        C^iM^j-1+\frac1qH_C(i,j)\le0.
\]
\end{lemma}

\begin{proof}
Put $\delta=C-1$ and $\rho=j/q$.  Then $0\le \rho\le1$, and the assumption gives $M^q\le p+1-pC=1-p\delta$. Since $M^q>0$, we have $1-p\delta>0$. We use the elementary fact that, for $0\le\rho\le1$, the function $x\mapsto x^\rho$ is concave on $(0,\infty)$; hence $x^\rho\le 1+\rho(x-1)$ for every $x>0$.  Applying this with $x=1-p\delta$, we get $(1-p\delta)^\rho\le 1-\rho p\delta$. Therefore $M^j=(M^q)^\rho\le (1-p\delta)^\rho\le 1-\rho p\delta$, and so 
\begin{equation}\label{euq-cm}
C^iM^j-1\le C^i-1-\rho p\delta C^i.
\end{equation}
Since $j=q\rho$, it follows from \eqref{hcij} that 
\begin{equation}\label{qhc}
\frac1qH_C(i,j)=\delta\bigl((p-i)\rho C^i-i(1-\rho)C^{i-1}\bigr). 
\end{equation}
Combining \eqref{euq-cm} and \eqref{qhc} gives
\[
\begin{aligned}
C^iM^j-1+\frac1qH_C(i,j)
&\le C^i-1-\rho p\delta C^i
+\delta\bigl((p-i)\rho C^i-i(1-\rho)C^{i-1}\bigr)  \\
&= C^i-1-i\rho\delta C^i-i(1-\rho)\delta C^{i-1}  \\
&= C^i-1-i\delta C^{i-1}(\rho C+1-\rho)  \\
&= C^i-1-i(C-1)C^{i-1}\bigl(1+\rho(C-1)\bigr).
\end{aligned}
\]
It remains to show that the last expression is nonpositive. If $i=0$, this is immediate. Assume $i\ge1$.
First suppose that $C\ge1$. Then $C^i-1=(C-1)(1+C+\cdots+C^{i-1})\le i(C-1)C^{i-1}$. Since $1+\rho(C-1)\ge1$, it follows that $C^i-1\le i(C-1)C^{i-1}\bigl(1+\rho(C-1)\bigr)$, as required.

It remains to consider $0<C\le1$.  Write $\eta=1-C$. Then $0\le\eta<1$ and $1+\rho(C-1)=1-\rho\eta\in[0,1]$. The desired inequality is equivalent to $1-C^i\ge i(1-C)C^{i-1}\bigl(1-\rho(1-C)\bigr)$. But $1-C^i=(1-C)(1+C+\cdots+C^{i-1})\ge i(1-C)C^{i-1}$, because $0<C\le1$ implies $1,C,\ldots,C^{i-1}\ge C^{i-1}$. Since $1-\rho(1-C)\le1$, the desired inequality follows. This proves $C^iM^j-1+\frac1qH_C(i,j)\le0$ in all cases.
\end{proof}

The next lemma is the simultaneous form of the preceding estimate. It is the point at which Balancer chooses the whole $p$-set at once.

\begin{lemma}\label{lem2-2}
Let $X$ be a finite set with $|X|\ge p+q$.  Let $\mathcal I$ be a finite index set. For each $r\in\mathcal I$, let $E_r\subseteq X$, let $w_r\ge0$, and let $C_r,M_r>0$ satisfy
$pC_r+M_r^q\le p+1$.
Let $S$ be a $p$-subset of $X$ such that
\begin{equation}
\label{eq:selection-objective}
\sum_{r\in\mathcal I}w_r C_r^{|E_r\cap S|}=\min \left\{\sum_{r\in\mathcal I}w_r C_r^{|E_r\cap A|}\,|\, A\subseteq X, |A|=p\right\}
\end{equation}
over all $p$-subsets of $X$.  Then for every $q$-subset
$T\subseteq X\setminus S$,
\[
        \sum_{r\in\mathcal I}
        w_r\left(C_r^{|E_r\cap S|}M_r^{|E_r\cap T|}-1\right)
        \le0.
\]
\end{lemma}

\begin{proof}
Let $F(S)=\sum_{r\in\mathcal I}w_r C_r^{|E_r\cap S|}$.
Fix a $q$-subset $T\subseteq X\setminus S$. For $x\in S$ and $y\in T$, write
        $S-x+y=(S\setminus\{x\})\cup\{y\}$, which is again a $p$-subset of $X$.  
Let
\begin{equation}\label{delta}
\Delta
        =
        \sum_{x\in S}\sum_{y\in T}
        \bigl(F(S-x+y)-F(S)\bigr).
\end{equation}
Since $F(S)=\min \left\{F(V)\,|\, V\subseteq X, |V|=p\right\}$, we have
        $F(S-x+y)-F(S)\ge0$ for $x\in S$ and $y\in T$.
Therefore, $\Delta\ge0$.
We now compute the same quantity by expanding the definition of $F$. Since
\[
        F(A)=\sum_{r\in\mathcal I}w_r C_r^{|E_r\cap A|}
        \qquad(A\subseteq X,\ |A|=p),
\]
it follows from \eqref{delta} that
\begin{equation}\label{euquation-2-2}
\begin{aligned}
        \Delta
        &=
        \sum_{x\in S}\sum_{y\in T}
        \sum_{r\in\mathcal I}
        w_r
        \left(
        C_r^{|E_r\cap(S-x+y)|}
        -
        C_r^{|E_r\cap S|}
        \right)  \\[0.2cm]
        &=
        \sum_{r\in\mathcal I}
        w_r
        \sum_{x\in S}\sum_{y\in T}
        \left(
        C_r^{|E_r\cap(S-x+y)|}
        -
        C_r^{|E_r\cap S|}
        \right).
\end{aligned}
\end{equation}
Fix $r\in\mathcal I$, and let
$i=|E_r\cap S|,
 j=|E_r\cap T|$.
We compute the inner double sum for this fixed $r$. If
$x\notin E_r$ and $y\in E_r$, then the exponent
$|E_r\cap S|$ increases from $i$ to $i+1$, and the change is
$C_r^{i+1}-C_r^i=(C_r-1)C_r^i$.
There are $(p-i)j$ ordered pairs $(x,y)\in S\times T$ of this type.
If $x\in E_r$ and $y\notin E_r$, then the exponent decreases from $i$ to
$i-1$, and the change is
$C_r^{i-1}-C_r^i=-(C_r-1)C_r^{i-1}$.
There are $i(q-j)$ ordered pairs of this type.
In the two remaining cases, namely when both $x,y$ lie in $E_r$ or when
neither of them lies in $E_r$, the number $|E_r\cap S|$ does not change,
and the contribution is zero.  Hence the inner double sum for the fixed
index $r$ is
$(p-i)j(C_r-1)C_r^i
        -
        i(q-j)(C_r-1)C_r^{i-1}$.
Thus
\begin{equation}\label{e2-3}
\begin{aligned}
        &\sum_{x\in S}\sum_{y\in T}
        \left(
        C_r^{|E_r\cap(S-x+y)|}
        -
        C_r^{|E_r\cap S|}
        \right)       \\
        &\qquad =
        (C_r-1)
        \Bigl((p-i)jC_r^i-i(q-j)C_r^{i-1}\Bigr)
        =
        H_{C_r}(i,j).
\end{aligned}
\end{equation}
Since $i=|E_r\cap S|$ and $j=|E_r\cap T|$, it follows from \eqref{euquation-2-2} and \eqref{e2-3} that 
$$\Delta
        =
        \sum_{r\in\mathcal I}
        w_r
        H_{C_r}(|E_r\cap S|,|E_r\cap T|).$$
Combining this identity with $\Delta\ge0$ gives
$$\sum_{r\in\mathcal I}
        w_r H_{C_r}(|E_r\cap S|,|E_r\cap T|)
        \ge0.$$
Applying Lemma \ref{lem-1} to each index $r$, with
$C=C_r$, $M=M_r$, $i=|E_r\cap S|$, and
        $j=|E_r\cap T|$,
we obtain
\[
        C_r^{|E_r\cap S|}M_r^{|E_r\cap T|}-1
        \le
        -\frac1q H_{C_r}(|E_r\cap S|,|E_r\cap T|).
\]
Multiplying by $w_r$ and summing over $r\in\mathcal I$ gives
\[
\sum_{r\in\mathcal I}
        w_r\left(C_r^{|E_r\cap S|}
        M_r^{|E_r\cap T|}-1\right)  \le
        -\frac1q
        \sum_{r\in\mathcal I}
        w_r H_{C_r}(|E_r\cap S|,|E_r\cap T|)
        \le0.
\]
This proves the lemma.
\end{proof}

Recall that
$R^+_{p,q}(t)$ is the positive solution of
\eqref{eq:R-plus}, and
 $R^-_{p,q}(s)$ is the positive solution of
\eqref{eq:R-minus}.
We now record the estimates for $R^+_{p,q}(t)$ and $R^-_{p,q}(s)$.

\begin{lemma}\label{fact:R-basic}
For $t>0$ and $s>0$,
$e^{-pt}<R^+_{p,q}(t)<e^{qt}$ and
$e^{-qs}<R^-_{p,q}(s)<e^{ps}$.
Moreover,
$R^+_{p,q}(t)>1$ and
        $R^-_{p,q}(s)>1$.
\end{lemma}
\begin{proof}
First consider the defining equation \eqref{eq:R-plus} for $R^+_{p,q}(t)$,
If $R\le e^{-pt}$, then, by \eqref{eq:R-plus}, we have
$\left(\frac{e^{-pt}}{R}\right)^q\ge1$
and
\[
        p\frac{e^{qt}}{R}
        \ge
        p e^{(p+q)t}
        >p.
\]
Hence the left hand side of \eqref{eq:R-plus} is larger than $p+1$, and so the positive solution cannot satisfy $R\le e^{-pt}$.  Therefore, $R^+_{p,q}(t)>e^{-pt}$.
If $R\ge e^{qt}$, then
$p\frac{e^{qt}}{R}\le p$
and
\[
        \left(\frac{e^{-pt}}{R}\right)^q
        \le
        e^{-q(p+q)t}
        <1.
\]
Hence the left hand side of \eqref{eq:R-plus} is smaller than $p+1$, and so the positive solution cannot satisfy $R\ge e^{qt}$. Therefore,
$R^+_{p,q}(t)<e^{qt}$.
This proves
$e^{-pt}<R^+_{p,q}(t)<e^{qt}$.

The estimates for $R^-_{p,q}(s)$ are obtained in the same way from \eqref{eq:R-minus}
Indeed, on the left hand side of \eqref{eq:R-minus}, if $R\le e^{-qs}$, then the first term is at least $p$ and the second term is larger than $1$; if $R\ge e^{ps}$, then the second term is at most $1$ and the first term is smaller than $p$.  Hence, $e^{-qs}<R^-_{p,q}(s)<e^{ps}$.

It remains to show that the two normalizing factors are larger than one.  Substituting $R=1$ in the defining equation for $R^+_{p,q}(t)$ gives
$f(t)=p e^{qt}+e^{-pqt}$.
Then, $f(0)=p+1$, $f'(0)=0$,
and
\[
        f''(t)=p q^2 e^{qt}+p^2q^2 e^{-pqt}>0
        \qquad(t\ge0).
\]
Thus $f(t)>p+1$ for every $t>0$. Since the left hand side of \eqref{pplus} for $R^+_{p,q}(t)$ is strictly decreasing as a function of $R$, the positive solution satisfies
$R^+_{p,q}(t)>1$.
Similarly, substituting $R=1$ in \eqref{eq:R-minus} for $R^-_{p,q}(s)$ gives
$g(s)=p e^{-qs}+e^{pqs}$.
We have $g(0)=p+1$,
        $g'(0)=0,$
and
\[
        g''(s)=p q^2 e^{-qs}+p^2q^2 e^{pqs}>0
        \qquad(s\ge0).
\]
Hence $g(s)>p+1$ for every $s>0$.  Since the left hand side of the defining equation for $R^-_{p,q}(s)$ is strictly decreasing as a function of $R$, we obtain
$R^-_{p,q}(s)>1$.
\end{proof}

We now give the proof of Theorem \ref{thm:main}.
\begin{proof}[Proof of Theorem~\ref{thm:main}]
 Let
$N=|V|/(p+q)$
be the number of complete rounds.  For each $i=0,1,\ldots,N$, let $B_i$ and $U_i$ be the sets already claimed by Balancer and Unbalancer after the first $i$ complete rounds.  Thus
$B_0=U_0=\emptyset$,
$B_N=B$,
and $U_N=U$.
Let
$X_i=V\setminus (B_i\cup U_i)$
be the set of vertices which are still unclaimed after the first $i$ complete rounds.
For an edge $e$, define the discrepancy after round $i$ by
$D_{e,i}=q|B_i\cap e|-p|U_i\cap e|$
and let
$c_{e,i}=|(B_i\cup U_i)\cap e|$
be the number of already claimed vertices of $e$ after round $i$. Thus $|e|-c_{e,i}$ is the number of vertices of $e$ which are still unclaimed after round $i$. Note that $D_{e,N}=D_e$ and $c_{e,N}=|e|$.

For each edge $e$ and each $i=0,1,\ldots,N$, we define two quantities, one for the upper deviation and one for the lower deviation:
\[
        \Phi_{e,i}^+
        =
        e^{-t_eL_e^+}
        \bigl(R^+_{p,q}(t_e)\bigr)^{|e|-c_{e,i}}
        e^{t_eD_{e,i}}
\]
and
\[
        \Phi_{e,i}^-
        =
        e^{-s_eL_e^-}
        \bigl(R^-_{p,q}(s_e)\bigr)^{|e|-c_{e,i}}
        e^{-s_eD_{e,i}}.
\]
Let
\begin{equation}\label{equphi-indexed}
    \Phi_i=\sum_{e\in\mathcal E}(\Phi_{e,i}^++\Phi_{e,i}^-).
\end{equation}
At the beginning of the game, $D_{e,0}=0$ and $c_{e,0}=0$ for every edge $e$. Hence the initial value of $\Phi_0$ is
\[
        \sum_{e\in\mathcal E}
        e^{-t_eL_e^+}\bigl(R^+_{p,q}(t_e)\bigr)^{|e|}
        +
        \sum_{e\in\mathcal E}
        e^{-s_eL_e^-}\bigl(R^-_{p,q}(s_e)\bigr)^{|e|}.
\]
By \eqref{eq:main-condition}, this initial value is at most $1$, that is,
$\Phi_0\le1$.

It remains to describe Balancer's strategy and to show that the total potential does not increase after each complete round.  Consider the beginning of round $i+1$, where $0\le i\le N-1$.  The set of still unclaimed vertices is $X_i$.  Since the game is played in complete rounds and has not ended, $|X_i|\ge p+q$.

For the upper quantity $\Phi_{e,i}^+$, let us first consider the effect of one Balancer hit in $e$. Suppose that Balancer claims an unclaimed vertex $x\in e$. Then $|B_i\cap e|$ increases by $1$, while $|U_i\cap e|$ does not change. Since
$D_{e,i}=q|B_i\cap e|-p|U_i\cap e|$,
the value of $D_{e,i}$ increases by $q$. At the same time, the number $c_{e,i}$ of already claimed vertices of $e$ increases by $1$, and hence the remaining number $|e|-c_{e,i}$ decreases by $1$. Therefore $\Phi_{e,i}^+$ is multiplied by
\[
        \frac{
        e^{-t_eL_e^+}
        \bigl(R^+_{p,q}(t_e)\bigr)^{|e|-c_{e,i}-1}
        e^{t_e(D_{e,i}+q)}
        }{
        e^{-t_eL_e^+}
        \bigl(R^+_{p,q}(t_e)\bigr)^{|e|-c_{e,i}}
        e^{t_eD_{e,i}}
        }
        =
        \frac{e^{qt_e}}{R^+_{p,q}(t_e)}.
\]
Thus the multiplier caused by one Balancer hit in $e$ is
\[
        B_e^+
        =
        \frac{e^{qt_e}}{R^+_{p,q}(t_e)}.
\]
Similarly, if Unbalancer claims an unclaimed vertex of $e$, then $|U_i\cap e|$ increases by $1$, while $|B_i\cap e|$ does not change.  Hence $D_{e,i}$ decreases by $p$, and again $|e|-c_{e,i}$ decreases by $1$. Therefore the multiplier caused by one Unbalancer hit in $e$ is
\[
        U_e^+
        =
        \frac{
        e^{-t_eL_e^+}
        \bigl(R^+_{p,q}(t_e)\bigr)^{|e|-c_{e,i}-1}
        e^{t_e(D_{e,i}-p)}
        }{
        e^{-t_eL_e^+}
        \bigl(R^+_{p,q}(t_e)\bigr)^{|e|-c_{e,i}}
        e^{t_eD_{e,i}}
        }
        =
        \frac{e^{-pt_e}}{R^+_{p,q}(t_e)}.
\]
By the defining equation \eqref{eq:R-plus}, we have
$pB_e^+ +(U_e^+)^q=p+1$.
Moreover, Lemma~\ref{fact:R-basic} gives
$B_e^+>1$ and
        $0<U_e^+<1$.
Thus, for the upper deviation, a Balancer hit is the increasing factor and an Unbalancer hit is the decreasing factor.

For the lower quantity $\Phi_{e,i}^-$, the signs are reversed. A Balancer hit in $e$ increases $D_{e,i}$ by $q$, so the factor $e^{-s_eD_{e,i}}$ is multiplied by $e^{-qs_e}$.  Taking also the decrease of $|e|-c_{e,i}$ into account, the Balancer-hit multiplier is
\[
        U_e^-
        =
        \frac{e^{-qs_e}}{R^-_{p,q}(s_e)}.
\]
An Unbalancer hit in $e$ decreases $D_{e,i}$ by $p$, so $e^{-s_eD_{e,i}}$ is multiplied by $e^{ps_e}$; the corresponding multiplier is
\[
        B_e^-
        =
        \frac{e^{ps_e}}{R^-_{p,q}(s_e)}.
\]
By \eqref{eq:R-minus},
$pU_e^- +(B_e^-)^q=p+1$.
Again, by Lemma~\ref{fact:R-basic},
$0<U_e^-<1$ and
        $B_e^->1$.
Thus, for the lower deviation, a Balancer hit is the decreasing factor and an Unbalancer hit is the increasing factor.

Balancer now chooses a $p$-subset $S_{i+1}\subseteq X_i$, where $0\le i\le N-1$, minimizing
\begin{equation}
\label{eq:strategy-objective-indexed}
        \sum_{e\in\mathcal E}
        \Phi_{e,i}^+(B_e^+)^{|e\cap S|}
        +
        \sum_{e\in\mathcal E}
        \Phi_{e,i}^-(U_e^-)^{|e\cap S|}
\end{equation}
among all $p$-subsets $S\subseteq X_i$. After this choice, Unbalancer chooses an arbitrary $q$-subset $T_{i+1}\subseteq X_i\setminus S_{i+1}$.

We compare the value of $\Phi_i$ before round $i+1$ and the value of $\Phi_{i+1}$ after this complete round. Fix an edge $e$ and put
\[
        a=|e\cap S_{i+1}|,
        \qquad
        b=|e\cap T_{i+1}|.
\]
After the round, the number $c_{e,i}$ has increased by $a+b$, and hence
$c_{e,i+1}=c_{e,i}+a+b$.
Since each Balancer vertex contributes $q$ to $D_{e,i}$ and each Unbalancer vertex contributes $-p$ to $D_{e,i}$, it follows from \eqref{eq:D} that
$D_{e,i+1}=D_{e,i}+aq-bp$.
Consequently, the new upper quantity is
\begin{equation}\label{equ-2-5-indexed}
\begin{aligned}
        \Phi_{e,i+1}^+
        &=
        e^{-t_eL_e^+}
        \bigl(R^+_{p,q}(t_e)\bigr)^{|e|-c_{e,i}-a-b}
        e^{t_e(D_{e,i}+aq-bp)}  \\[0.2cm]
        &=
        \Phi_{e,i}^+
        \left(\frac{e^{qt_e}}{R^+_{p,q}(t_e)}\right)^a
        \left(\frac{e^{-pt_e}}{R^+_{p,q}(t_e)}\right)^b  \\[0.2cm]
        &=
        \Phi_{e,i}^+(B_e^+)^a(U_e^+)^b.
\end{aligned}
\end{equation}
In the same way,
\begin{equation}\label{equ-2-6-indexed}
\begin{aligned}
        \Phi_{e,i+1}^-
        &=
        e^{-s_eL_e^-}
        \bigl(R^-_{p,q}(s_e)\bigr)^{|e|-c_{e,i}-a-b}
        e^{-s_e(D_{e,i}+aq-bp)} \\[0.2cm]
        &=
        \Phi_{e,i}^-
        \left(\frac{e^{-qs_e}}{R^-_{p,q}(s_e)}\right)^a
        \left(\frac{e^{ps_e}}{R^-_{p,q}(s_e)}\right)^b  \\[0.2cm]
        &=
        \Phi_{e,i}^-(U_e^-)^a(B_e^-)^b.
\end{aligned}
\end{equation}

We apply Lemma~\ref{lem2-2} with ground set $X_i$ and with
index set
$\mathcal I=\mathcal E\times\{+,-\}$.
Thus, for each edge $e\in\mathcal E$, we introduce two indexed terms:
one upper term, denoted by $(e,+)$, and one lower term, denoted by
$(e,-)$.
For the upper term we take
$E_{(e,+)}=e\cap X_i,
        w_{(e,+)}=\Phi_{e,i}^+,
        C_{(e,+)}=B_e^+$,
and
$M_{(e,+)}=U_e^+$.
For the lower term we take
$E_{(e,-)}=e\cap X_i,
        w_{(e,-)}=\Phi_{e,i}^-,
        C_{(e,-)}=U_e^-$,
and
        $M_{(e,-)}=B_e^-$.
Since each edge gives exactly two such terms, the index set has size
$|\mathcal I|=2|\mathcal E|$.
The required hypotheses of Lemma \ref{lem2-2} hold because
$ pB_e^+ +(U_e^+)^q=p+1$
and
 $pU_e^- +(B_e^-)^q=p+1$.
It remains to check that Balancer's chosen set $S_{i+1}$ is a minimizer of the function in \eqref{eq:selection-objective}.  For a $p$-subset $P\subseteq X_i$, the function appearing in Lemma~\ref{lem2-2} is
\[
        \sum_{r\in\mathcal I} w_r C_r^{|E_r\cap P|}.
\]
Since
 $\mathcal I=\mathcal E\times\{+,-\},$
this sum is
\[
        \sum_{e\in\mathcal E}
        w_{(e,+)} C_{(e,+)}^{|E_{(e,+)}\cap P|}
        +
        \sum_{e\in\mathcal E}
        w_{(e,-)} C_{(e,-)}^{|E_{(e,-)}\cap P|}.
\]
Using the definitions
$E_{(e,+)}=e\cap X_i, w_{(e,+)}=\Phi_{e,i}^+$,
        $C_{(e,+)}=B_e^+$,
and
$E_{(e,-)}=e\cap X_i, w_{(e,-)}=\Phi_{e,i}^-,
        C_{(e,-)}=U_e^-$,
and using $P\subseteq X_i$, we have
$|E_{(e,+)}\cap P|=|e\cap P|$
and
$ |E_{(e,-)}\cap P|=|e\cap P|$.
Therefore
\[
        \sum_{r\in\mathcal I} w_r C_r^{|E_r\cap P|}
        =
        \sum_{e\in\mathcal E}
        \Phi_{e,i}^+(B_e^+)^{|e\cap P|}
        +
        \sum_{e\in\mathcal E}
        \Phi_{e,i}^-(U_e^-)^{|e\cap P|}.
\]
The right hand side is precisely the quantity which Balancer minimizes in
\eqref{eq:strategy-objective-indexed}.  Hence the $p$-set $S_{i+1}$ chosen by Balancer satisfies the minimizing condition in Lemma~\ref{lem2-2}.
Thus, for the $q$-set $T_{i+1}$ chosen by Unbalancer, Lemma \ref{lem2-2} gives
$$
    \sum_{e\in\mathcal E}
        \Phi_{e,i}^+
        \left((B_e^+)^{|e\cap S_{i+1}|}
        (U_e^+)^{|e\cap T_{i+1}|}-1\right)+
        \sum_{e\in\mathcal E}
        \Phi_{e,i}^-
        \left((U_e^-)^{|e\cap S_{i+1}|}
        (B_e^-)^{|e\cap T_{i+1}|}-1\right)
        \le0.
$$

For each edge $e$, by \eqref{equ-2-5-indexed} and \eqref{equ-2-6-indexed}, the change of the upper potential of $e$ during round $i+1$ is
\[
        \Phi_{e,i+1}^+-\Phi_{e,i}^+
        =
        \Phi_{e,i}^+
        \left(
        (B_e^+)^{|e\cap S_{i+1}|}
        (U_e^+)^{|e\cap T_{i+1}|}
        -1
        \right),
\]
and the change of the lower potential of $e$ is
\[
        \Phi_{e,i+1}^--\Phi_{e,i}^-
        =
        \Phi_{e,i}^-
        \left(
        (U_e^-)^{|e\cap S_{i+1}|}
        (B_e^-)^{|e\cap T_{i+1}|}
        -1
        \right).
\]
Summing these two identities over all edges $e\in\mathcal E$ and using \eqref{equphi-indexed}, we obtain
\[
\begin{aligned}
        \Phi_{i+1}-\Phi_i
        &=
        \sum_{e\in\mathcal E}
        \Phi_{e,i}^+
        \left(
        (B_e^+)^{|e\cap S_{i+1}|}
        (U_e^+)^{|e\cap T_{i+1}|}
        -1
        \right) \\
        &\quad+
        \sum_{e\in\mathcal E}
        \Phi_{e,i}^-
        \left(
        (U_e^-)^{|e\cap S_{i+1}|}
        (B_e^-)^{|e\cap T_{i+1}|}
        -1
        \right).
\end{aligned}
\]
The right hand side is exactly the expression bounded above by $0$ by Lemma~\ref{lem2-2}.  Therefore
$\Phi_{i+1}-\Phi_i\le0$,
and hence
$\Phi_{i+1}\le\Phi_i$.
Thus the total potential does not increase after each complete round.

Since initially $\Phi_0\le1$, induction over the complete rounds gives
$\Phi_N\le1$
at the end of the game.
At the end of the game all vertices have been claimed, so $c_{e,N}=|e|$ for every edge $e$. Therefore
\[
        \Phi_{e,N}^+=e^{-t_eL_e^+}e^{t_eD_{e,N}}
        =
        e^{t_e(D_{e,N}-L_e^+)}
\]
and
\[
        \Phi_{e,N}^-=e^{-s_eL_e^-}e^{-s_eD_{e,N}}
        =
        e^{-s_e(D_{e,N}+L_e^-)}.
\]
Suppose that $D_{e,N}>L_e^+$ for some edge $e$.  Then
\[
        \Phi_{e,N}^+=e^{t_e(D_{e,N}-L_e^+)}>1.
\]
Since all terms in $\Phi_N$ are nonnegative, this implies $\Phi_N>1$, contradicting the bound obtained above. Hence
        $D_{e,N}\le L_e^+$ for $e\in\mathcal E$.
Similarly, if $D_{e,N}<-L_e^-$ for some edge $e$, then
\[
        \Phi_{e,N}^-=e^{-s_e(D_{e,N}+L_e^-)}>1,
\]
again contradicting $\Phi_N\le1$.  Therefore
$D_{e,N}\ge -L_e^-
        (e\in\mathcal E)$.
Combining the two inequalities gives
$ -L_e^-\le D_{e,N}\le L_e^+
  (e\in\mathcal E)$.
Finally, since
\[
        D_{e,N}=q|B\cap e|-p|U\cap e|=(p+q)|B\cap e|-p|e|,
\]
this is equivalent to
\[
        -\frac{L_e^-}{p+q}
        \le
        |B\cap e|-\frac{p}{p+q}|e|
        \le
        \frac{L_e^+}{p+q}.
\]
This proves Theorem~\ref{thm:main}.
\end{proof}

\begin{remark}
The normalizing equations defining $R^+_{p,q}$ and $R^-_{p,q}$ are forced by a simple matching obstruction.  Consider $p+1$ pairwise disjoint $q$-edges with equal initial weights.  After Balancer has claimed $p$ vertices, at least one of these edges is untouched, and Unbalancer may claim all $q$ vertices of this edge. Hence any separable one-round multiplicative potential with Balancer multiplier $B$ and Unbalancer multiplier $A$ must satisfy \(pB+A^q\le p+1\). For the upper-tail exponential potential, where \(B=e^{qt}/R\) and \(A=e^{-pt}/R\), this necessary condition becomes \(p e^{qt}/R+(e^{-pt}/R)^q\le p+1\).  Thus the smallest admissible normalization is precisely \(R^+_{p,q}(t)\). The lower-tail normalization \(R^-_{p,q}(s)\) is obtained in the same way, with the roles of the two multipliers reversed.
\end{remark}

Theorem~\ref{thm:main} recovers the two standard comparison results stated in
Theorems~\ref{thm:AKSS} and~\ref{thm:beck-uniform}.  
We first derive Theorem~\ref{thm:AKSS} from Theorem ~\ref{thm:main} by taking $p=q=1$.

\begin{proof}[Proof of Theorems~\ref{thm:AKSS}]
Put $p=q=1$. Then the defining
relations \eqref{eq:R-plus} and \eqref{eq:R-minus} become
\[
        \frac{e^t}{R}+\frac{e^{-t}}{R}=2,
        \qquad
        \frac{e^{-s}}{R}+\frac{e^s}{R}=2,
\]
and hence
$R^+_{1,1}(t)=R^-_{1,1}(t)=\cosh t$.
Let $H=(V,\mathcal E)$ with
$\mathcal E=\{e_1,\ldots,e_m\}$, and assume that the hypothesis of
Theorem~\ref{thm:AKSS} holds.  For each $i$ take
\[
        L_{e_i}^+=L_{e_i}^-=b_i,
        \qquad
        t_{e_i}=s_{e_i}=\frac{b_i}{|e_i|}.
\]
If the hypothesis is non-vacuous, then $b_i>0$ for every nonempty edge, so
these parameters are admissible.  Using $\cosh x\le \exp(x^2/2)$, we get
\[
\begin{aligned}
        &\sum_{i=1}^m e^{-t_{e_i}b_i}
        \bigl(R^+_{1,1}(t_{e_i})\bigr)^{|e_i|}
        +
        \sum_{i=1}^m e^{-s_{e_i}b_i}
        \bigl(R^-_{1,1}(s_{e_i})\bigr)^{|e_i|}       \\
        &\qquad\le
        2\sum_{i=1}^m
        \exp\!\left\{-\frac{b_i^2}{|e_i|}
        +\frac{|e_i|}{2}\left(\frac{b_i}{|e_i|}\right)^2\right\}
        =
        2\sum_{i=1}^m
        \exp\!\left\{-\frac{b_i^2}{2|e_i|}\right\}
        \le 1.
\end{aligned}
\]
Thus the condition of Theorem~\ref{thm:main} is satisfied. Since for
$p=q=1$ we have
$D_e=|B\cap e|-|U\cap e|$, Theorem~\ref{thm:main} gives
$\bigl||B\cap e_i|-|U\cap e_i|\bigr|\le b_i$ for every $i$, which is
Theorem~\ref{thm:AKSS}.
\end{proof}

Then, we obtain Theorem~\ref{thm:beck-uniform} by Theorem ~\ref{thm:main}

\begin{proof}[Proof of Theorem ~\ref{thm:beck-uniform}]
Fix $p,q$, and let
$H=(V,\mathcal E)$ be $n$-uniform with $m=|\mathcal E|$. If necessary, add
dummy vertices belonging to no edge so that the total number of board vertices
is divisible by $p+q$; this does not affect any edge discrepancy.  We shall use
common parameters and common targets
$L_e^+=L_e^-=L$ and
$t_e=s_e=t
~ (e\in\mathcal E)$.
Write
$r^+(x)=\log R^+_{p,q}(x)$,
and $r^-(x)=\log R^-_{p,q}(x)$.
By the implicit function theorem applied to \eqref{eq:R-plus} and
\eqref{eq:R-minus} at $(x,r)=(0,0)$, the functions $r^+$ and $r^-$ are smooth
near $0$.  Differentiating the defining equations gives
$r^+(0)=r^-(0)=0$,
$(r^+)'(0)=(r^-)'(0)=0$,
and
\[
        (r^+)''(0)=(r^-)''(0)=
        \frac{p(p+1)q^2}{p+q}.
\]
Consequently, as $x\to0^+$,
\begin{equation}\label{eq:R-Taylor-for-corollary}
        \log R^+_{p,q}(x),\ \log R^-_{p,q}(x)
        \le
        \frac{p^2q^2}{p+q}x^2+O_{p,q}(x^3),
\end{equation}
since $p(p+1)/2\le p^2$.

Put
\[
        \Lambda=\log(2m),
        \qquad
        b=\frac{p^2q^2}{p+q}.
\]
We first consider the case where \(\sqrt{\Lambda/n}\) is bounded below by a
positive constant depending only on \(p,q\).  Since \(H=(V,\mathcal E)\) is
\(n\)-uniform, 
 for every \(e\in\mathcal E\),
$0\le |B\cap e|\le n,$
and therefore
\[
        \left|
        |B\cap e|-\frac{p}{p+q}n
        \right|
        \le n.
\]
On the other hand, in the present case, by choosing the constant
\(C=C(p,q)\) sufficiently large, we may ensure that
\[
        \frac{1}{p+q}
        2\sqrt{bn\Lambda}
        \left(1+C\sqrt{\frac{\Lambda}{n}}\right)
        \ge n.
\]
Consequently,
\[
        \left|
        |B\cap e|-\frac{p}{p+q}n
        \right|
        \le
        \frac{1}{p+q}
        2\sqrt{bn\Lambda}
        \left(1+C\sqrt{\frac{\Lambda}{n}}\right)
        \qquad(e\in\mathcal E).
\]
Thus the desired bound already holds in this case.  Hence, in the rest of
the proof, we may assume that \(\sqrt{\Lambda/n}\) is sufficiently small,
where the required smallness depends only on \(p,q\).
Thus we may assume from now on that \(\sqrt{\Lambda/n}\) is sufficiently small.
Choose a sufficiently large
constant $C=C(p,q)$ and set
\[
        L=2\sqrt{bn\Lambda}
        \left(1+C\sqrt{\frac{\Lambda}{n}}\right),
        \qquad
        t=\frac{L}{2bn}.
\]
Then $t$ is small, and \eqref{eq:R-Taylor-for-corollary} gives, for the same
choice $t=s$,
\[
\begin{aligned}
        -tL+n\log R^+_{p,q}(t)
        &\le
        -\frac{L^2}{2bn}
        +\frac{L^2}{4bn}
        +O_{p,q}(nt^3)       \\
        &=
        -\frac{L^2}{4bn}+O_{p,q}(nt^3)
        \le -\Lambda,
\end{aligned}
\]
and the same estimate holds with $R^-_{p,q}$ in place of $R^+_{p,q}$. Thus, we have 
\[
        e^{-tL}\bigl(R^+_{p,q}(t)\bigr)^n\le e^{-\Lambda},
        \qquad
        e^{-tL}\bigl(R^-_{p,q}(t)\bigr)^n\le e^{-\Lambda}.
\]
Therefore
\[
        \sum_{e\in\mathcal E} e^{-tL}
        \bigl(R^+_{p,q}(t)\bigr)^n
        +
        \sum_{e\in\mathcal E} e^{-tL}
        \bigl(R^-_{p,q}(t)\bigr)^n
        \le
        2m e^{-\Lambda}=1.
\]
Theorem~\ref{thm:main} now yields
$|D_e|\le L$ for every edge $e$.  Since
$D_e=(p+q)|B\cap e|-pn$, we obtain
\[
        \left|
        |B\cap e|-\frac{p}{p+q}n
        \right|
        \le
        \frac{L}{p+q}
        =
        \frac{\alpha n}{p+q},
\]
where
\[
        \alpha=\frac{L}{n}
        =
        \left(
        1+O_{p,q}\!\left(
        pq\sqrt{\frac{\log(2m)}{(p+q)n}}
        \right)
        \right)
        2pq
        \sqrt{\frac{\log(2m)}{(p+q)n}}.
\]
This is the claimed uniform biased, Beck-type consequence of
Theorem~\ref{thm:main}.
\end{proof}

\section{Proof of Theorem~\ref{thm:pseudorandom}}\label{sec:proof-pseudorandom}

We shall use the following local estimate for the functions
$R^+_{p,q}$ and $R^-_{p,q}$.

\begin{lemma}\label{lem:local-quadratic}
For fixed positive integers $p,q$, there exist constants
$A=A(p,q)>0$ and $\tau=\tau(p,q)>0$ such that
\[
        \log R^+_{p,q}(x)\le A x^2,
        \qquad
        \log R^-_{p,q}(x)\le A x^2
        \qquad(0\le x\le \tau).
\]
\end{lemma}

\begin{proof}
We prove the assertion for $R^+_{p,q}$; the proof for
$R^-_{p,q}$ is the same.  
Since $R^+_{p,q}(t)>0$, we write
$R^+_{p,q}(t)=e^{r(t)}$,
that is, $r(t)=\log R^+_{p,q}(t)$.
Substituting this into \eqref{eq:R-plus}, we obtain
\begin{equation}\label{equ3-1}
p e^{qt-r(t)}+e^{-pqt-qr(t)}=p+1.
\end{equation}
To justify the expansion of $r(t)$ near $t=0$, define
\[
        F(t,r)=p e^{qt-r}+e^{-pqt-qr}-(p+1).
\]
Then,
$F(0,0)=0$
and
\[
        \frac{\partial F}{\partial r}(0,0)
        =
        -p-q
        \ne0.
\]
Hence the equation $F(t,r)=0$ determines, for $t$ sufficiently close to $0$, a unique analytic function $r=r(t)$ with $r(0)=0$.  Therefore the Taylor expansion of $r(t)$ has the form
$r(t)=\alpha t+\beta t^2+O_{p,q}(t^3)$,
where $\alpha$ and $\beta$ are constants depending only on $p,q$.
Substituting this expansion into \eqref{equ3-1}
gives
\begin{equation}\label{equ3-2}
p e^{(q-\alpha)t-\beta t^2+O_{p,q}(t^3)}
        +
        e^{(-pq-q\alpha)t-q\beta t^2+O_{p,q}(t^3)}
        =
        p+1.
\end{equation}
Using $e^{at+O(t^2)}=1+at+O(t^2)$, the coefficient of $t$ in the left hand side is
        $p(q-\alpha)+(-pq-q\alpha)
        =
        -(p+q)\alpha$.
The right hand side $p+1$ has no term of order $t$, and hence this coefficient must be zero.  Therefore $\alpha=0$, and hence \eqref{equ3-2} becomes
\[
        p e^{qt-\beta t^2+O_{p,q}(t^3)}
        +
        e^{-pqt-q\beta t^2+O_{p,q}(t^3)}
        =
        p+1.
\]
Using
\[
        e^{at+bt^2+O(t^3)}
        =
        1+at+\left(b+\frac{a^2}{2}\right)t^2+O(t^3),
\]
we have
\begin{equation}\label{eq-p}
 p e^{qt-\beta t^2+O_{p,q}(t^3)}
    =
        p\left(
        1+qt+\left(\frac{q^2}{2}-\beta\right)t^2
        +O_{p,q}(t^3)
        \right),
\end{equation}
and
\begin{equation}\label{eq-q} e^{-pqt-q\beta t^2+O_{p,q}(t^3)}
        =
        1-pq\,t+
        \left(\frac{p^2q^2}{2}-q\beta\right)t^2
        +O_{p,q}(t^3).
\end{equation}
By \eqref{eq-p} and \eqref{eq-q}, we give
\[
\begin{aligned}
        p e^{qt-\beta t^2+O_{p,q}(t^3)}
        &+
        e^{-pqt-q\beta t^2+O_{p,q}(t^3)}  \\
        &=
        p+1+
        \left[
        p\left(\frac{q^2}{2}-\beta\right)
        +
        \left(\frac{p^2q^2}{2}-q\beta\right)
        \right]t^2
        +O_{p,q}(t^3).
\end{aligned}
\]
Since the right hand side of \eqref{equ3-1} is the constant $p+1$,
the coefficient of $t^2$ must be zero. Hence
\[
        p\left(\frac{q^2}{2}-\beta\right)
        +
        \left(\frac{p^2q^2}{2}-q\beta\right)
        =
        0.
\]
Equivalently,
\[
        \frac{p q^2(p+1)}{2}-(p+q)\beta=0,
\]
and therefore
\[
        \beta=\frac{p q^2(p+1)}{2(p+q)}.
\]
Consequently,
\[
        \log R^+_{p,q}(t)
        =
        r(t)
        =
        \frac{p q^2(p+1)}{2(p+q)}t^2
        +
        O_{p,q}(t^3).
\]
The same computation applied to \eqref{eq:R-minus} gives
\[
        \log R^-_{p,q}(s)
        =
        \frac{p q^2(p+1)}{2(p+q)}s^2+O_{p,q}(s^3).
\]
Choosing $A=A(p,q)$ sufficiently large and then choosing
$\tau=\tau(p,q)>0$ sufficiently small proves the lemma.
\end{proof}

\begin{proof}[Proof of Theorem~\ref{thm:pseudorandom}]
Let $A$ and $\tau$ be the constants from Lemma~\ref{lem:local-quadratic}.
We shall choose the constants in the statement in the following order.  First choose
$K=K(p,q)$ so large that
\[
        \frac{(p+q)^2K^2}{4A}\ge 3.
\]
Next choose $\eta_0=\eta_0(p,q)>0$ so small that
\[
        0<\eta_0\le \frac14
        \qquad\text{and}\qquad
        \frac{(p+q)\eta_0}{4A}\le \tau.
\]
Finally choose $C=C(p,q)$ sufficiently large.  During the proof we shall increase
$C$ several times, always depending only on $p,q$.

Let $0<\eta\le\eta_0$, and put
        $k=\lceil \eta n\rceil$.
By increasing $C$ if necessary, the assumption
$\eta^3 n\ge C\log(e/\eta)$
implies
        $\eta n\ge1$
        and hence 
        $k\le 2\eta n$.
It also implies that $n$ is sufficiently large for all estimates below.

We first prove the assertion in the complete-round formulation.  The board of the
auxiliary discrepancy game is the edge set of $K_n$, that is, $W=E(K_n)$.
A set claimed by Balancer in this auxiliary game will be regarded as the edge set
of the graph $G_B$.

We define an auxiliary hypergraph
$\mathcal H=(W,\mathcal F)$
as follows.  The first family of hyperedges consists of the stars
$F_v=\{vx:x\in[n]\setminus\{v\}\}
  (v\in[n])$.
Thus
        $|F_v|=n-1$.
The second family of hyperedges consists of bipartite edge sets
$F_{S,T}=\{xy:x\in S,\ y\in T\}$,
where $S,T\subseteq[n]$ are disjoint $k$-sets. Thus
        $|F_{S,T}|=k^2$.
We shall apply Theorem~\ref{thm:main} to this auxiliary hypergraph.

For every star $F_v$, set
\begin{equation}\label{3-5}
L_v^+=L_v^-=(p+q)K\sqrt{(n-1)\log(2n)}.
\end{equation}
We choose the same exponential parameter for the two tails:
\begin{equation}\label{3-6}
t_v=s_v=
        \frac{L_v^+}{2A(n-1)}
        =
        \frac{(p+q)K}{2A}
        \sqrt{\frac{\log(2n)}{n-1}}.
\end{equation}
After increasing $C$, the hypothesis on $\eta^3 n$ guarantees that
$n$ is large enough so that
$t_v=s_v\le\tau$
for $v\in[n]$.

By Lemma~\ref{lem:local-quadratic},
$\log R^+_{p,q}(t_v)\le A t_v^2$ and $\log R^-_{p,q}(s_v)\le A s_v^2$. 
Therefore, by \eqref{3-5} and \eqref{3-6}, each upper star summand in \eqref{eq:main-condition} is at most
\[
\begin{aligned}
        e^{-t_vL_v^+}
        \bigl(R^+_{p,q}(t_v)\bigr)^{n-1}
        &\le
        \exp\{-t_vL_v^+ + A(n-1)t_v^2\}  \\
        &=
        \exp\left\{
        -\frac{(L_v^+)^2}{4A(n-1)}
        \right\}  \\
        &=
        \exp\left\{
        -\frac{(p+q)^2K^2}{4A}\log(2n)
        \right\}.
\end{aligned}
\]
The same estimate holds for the lower star summand.  Since
\[
        \frac{(p+q)^2K^2}{4A}\ge3,
\]
the total contribution of the upper and lower tails of all stars is at most
$2n(2n)^{-3}<\frac14$
for $n\ge3$.

We now consider the bipartite hyperedges $F_{S,T}$.  For every ordered pair
of disjoint $k$-sets $S,T$, set
\begin{equation}\label{3-7}
L_{S,T}^+=L_{S,T}^-=(p+q)\frac{\eta}{2}k^2.
\end{equation}
Again choose the same exponential parameter for the two tails:
\begin{equation}\label{3-8}
t_{S,T}=s_{S,T}
        =
        \frac{L_{S,T}^+}{2A k^2}
        =
        \frac{(p+q)\eta}{4A}.
\end{equation}
The choice of $\eta_0$ gives
$t_{S,T}=s_{S,T}\le\tau$.
Thus Lemma~\ref{lem:local-quadratic} gives
$\log R^+_{p,q}(t_{S,T})\le A t_{S,T}^2$
and $\log R^-_{p,q}(s_{S,T})\le A s_{S,T}^2$.
Hence, by \eqref{3-7} and \eqref{3-8}, each upper bipartite summand is at most
\[
\begin{aligned}
        e^{-t_{S,T}L_{S,T}^+}
        \bigl(R^+_{p,q}(t_{S,T})\bigr)^{k^2}
        &\le
        \exp\{-t_{S,T}L_{S,T}^+ + A k^2t_{S,T}^2\}  \\
        &=
        \exp\left\{
        -\frac{(L_{S,T}^+)^2}{4Ak^2}
        \right\}  \\
        &=
        \exp\left\{
        -\frac{(p+q)^2}{16A}\eta^2k^2
        \right\}.
\end{aligned}
\]
The lower bipartite summand satisfies the same bound.  Put
\[
        \gamma=\frac{(p+q)^2}{16A}.
\]
The number of ordered pairs of disjoint $k$-subsets is at most
$\binom nk^2$.
Using
$\binom nk\le \left(\frac{en}{k}\right)^k$
and $k\ge \eta n$,
we get
\[
        \binom nk^2
        \le
        \left(\frac{en}{k}\right)^{2k}
        \le
        \left(\frac e\eta\right)^{2k}.
\]
Therefore the total contribution of all upper and lower bipartite tails is at most
$2\exp\left\{
        2k\log(e/\eta)-\gamma\eta^2k^2
        \right\}$.
Since $k\ge\eta n$, we have
$\eta^2k^2\ge \eta^3nk$.
Thus the exponent is at most
$2k\log(e/\eta)-\gamma\eta^3nk$.
By increasing $C$ so that
$\gamma C\ge5$,
the hypothesis
$ \eta^3 n\ge C\log(e/\eta)$
implies
$2k\log(e/\eta)-\gamma\eta^3nk
        \le
        -3k\log(e/\eta)$.
As $k\ge1$ and $\log(e/\eta)\ge1$, the total contribution of all bipartite
hyperedges is less than
$2e^{-3}<\frac14$.

Combining the star and bipartite estimates, the sum in \eqref{eq:main-condition}
for the auxiliary hypergraph is strictly smaller than $1$.  Theorem~\ref{thm:main}
therefore applies.  Hence Balancer has a strategy such that, for every vertex
$v\in[n]$,
\[
        \left|
        |B\cap F_v|-\rho |F_v|
        \right|
        \le
        \frac{L_v^+}{p+q}
        =
        K\sqrt{(n-1)\log(2n)}.
\]
Since $B$ is the edge set of $G_B$, we have
$|B\cap F_v|=d_{G_B}(v)$
and
$|F_v|=n-1$.
Thus
\[
        \left|
        d_{G_B}(v)-\rho(n-1)
        \right|
        \le
        K\sqrt{(n-1)\log(2n)}
        \qquad(v\in[n]).
\]
After increasing the constant $C=C(p,q)$ in the statement, this gives
\[
        \left|
        d_{G_B}(v)-\rho(n-1)
        \right|
        \le
        C\sqrt{n\log n}
        \qquad(v\in[n])
\]
in the complete-round setting.

Similarly, for every ordered pair of disjoint $k$-sets $S,T$, Theorem~\ref{thm:main}
gives
\[
        \left|
        |B\cap F_{S,T}|-\rho |F_{S,T}|
        \right|
        \le
        \frac{L_{S,T}^+}{p+q}
        =
        \frac{\eta}{2}k^2.
\]
Since
$|B\cap F_{S,T}|=e_{G_B}(S,T)$
and $|F_{S,T}|=k^2$,
we obtain
\[
        \left|
        e_{G_B}(S,T)-\rho k^2
        \right|
        \le
        \frac{\eta}{2}k^2
\]
for all disjoint $k$-sets $S,T$.

It remains to discuss the possible last incomplete round.  If the board size is not
divisible by $p+q$, or if the rules force a final incomplete round, then at most
$p+q-1$ final claims are outside the complete-round analysis.  Such claims can change
the degree of any vertex by at most $p+q-1$, and can change the value of
$e_{G_B}(S,T)$ for any fixed pair $S,T$ by at most $p+q-1$.  Hence the degree estimate
gets an additional $O_{p,q}(1)$ term.  For the bipartite estimates, by increasing
$C$ once more we may assume
$p+q-1\le \frac{\eta}{2}k^2$.
Indeed, since $k\ge\eta n$, we have
$ \eta k^2\ge \eta^3 n^2\ge \eta^3 n\ge C\log(e/\eta)$,
and the right hand side can be made larger than any fixed constant depending only on
$p,q$.  Therefore, in the original game, for all disjoint $k$-sets $S,T$,
$\left|
        e_{G_B}(S,T)-\rho k^2
        \right|
        \le
        \eta k^2$.

Finally we pass from $k$-sets to arbitrary larger sets.  Let $S,T\subseteq[n]$ be
disjoint sets with
        $|S|,|T|\ge\eta n$.
Since $k=\lceil\eta n\rceil$, we have $k\le |S|$ and $k\le |T|$.  Choose uniformly
random $k$-subsets
$S'\subseteq S$ and
        $T'\subseteq T$.
For every choice of $S'$ and $T'$, the sets are disjoint, and hence the estimate just
proved gives
\begin{equation}\label{3-9}
\left|
        e_{G_B}(S',T')-\rho k^2
        \right|
        \le
        \eta k^2.
\end{equation}
Now each edge between $S$ and $T$ is
chosen in $e_{G_B}(S',T')$ with probability
\[
        \frac{k}{|S|}\cdot\frac{k}{|T|}
        =
        \frac{k^2}{|S||T|}.
\]
Therefore
\[
        \mathbb E[\,e_{G_B}(S',T')]
        =
        \frac{k^2}{|S||T|}e_{G_B}(S,T).
\]
Since \eqref{3-9}
holds for every choice of the \(k\)-subsets \(S'\subseteq S\) and
\(T'\subseteq T\), it also holds after averaging over all such choices. Hence
\[
        \left|
        \mathbb E [e_{G_B}(S',T')]-\rho k^2
        \right|
        \le
        \eta k^2 .
\]
Using
\[
        \mathbb E [e_{G_B}(S',T')]
        =
        \frac{k^2}{|S||T|}e_{G_B}(S,T),
\]
we obtain
\[
        \left|
        \frac{k^2}{|S||T|}e_{G_B}(S,T)-\rho k^2
        \right|
        \le
        \eta k^2 .
\]
Dividing by $k^2$ gives
\[
        \left|
        \frac{e_{G_B}(S,T)}{|S||T|}-\rho
        \right|
        \le
        \eta.
\]
Together with the degree estimate, this proves Theorem~\ref{thm:pseudorandom}.
\end{proof}

\section*{Acknowledgment}

We would like to express our gratitude to Professor Noga Alon for his reading and comments of this paper.


\begin{thebibliography}{99}

\bibitem{AKSS}
N. Alon, M. Krivelevich, J. Spencer and T. Szab\'o,
 Discrepancy games,
 \emph{Electronic J. Combin.} 12 (2005), Research Paper 51.

\bibitem{Beck1981Roth}
J. Beck,
 Roth's estimate of the discrepancy of integer sequences is nearly sharp,
 \emph{Combinatorica} 1 (1981), 319--325.


\bibitem{Beck81}
J. Beck,
 Van der Waerden and Ramsey type games,
 \emph{Combinatorica} 1 (1981), 103--116.

\bibitem{Beck}
J. Beck,
 \emph{Combinatorial Games: Tic-Tac-Toe Theory},
 Cambridge University Press, 2010.

\bibitem{BeckFiala1981}
J. Beck and T. Fiala,
 ``Integer-making'' theorems,
 \emph{Discrete Appl. Math.} 3(1)(1981), 1--8.

\bibitem{Chazelle}
B. Chazelle,
 \emph{The Discrepancy Method: Randomness and Complexity},
 Cambridge University Press, Cambridge, 2013.

\bibitem{vanderCorput1935}
J. G. van der Corput,
Verteilungsfunktionen I,
\emph{Proc. Kon. Ned. Akad. v. Wetensch.} 38 (1935), 813--821.

\bibitem{ErdosSelfridge}
P. Erd\H{o}s and J. L. Selfridge,
 On a combinatorial game,
 \emph{J. Combin. Theory, Ser. A} 14(3) (1973), 298--301.

\bibitem{FKPS}
A. Frieze, M. Krivelevich, O. Pikhurko and T. Szab\'o,
 The game of JumbleG,
 \emph{Combin. Prob. Comput.} 14 (2005), 783--793.

\bibitem{HJ63}
A. W. Hales and R. I. Jewett,
 Regularity and positional games,
 \emph{Trans. Amer. Math. Soc.} 106(2) (1963), 222--229.

\bibitem{HKSS14}
D. Hefetz, M. Krivelevich, M. Stojakovi\'c and T. Szab\'o,
 \emph{Positional Games},
 Oberwolfach Seminars, Vol. 44, Birkh\"auser/Springer, Basel, 2014.


\bibitem{HKS}
D. Hefetz, M. Krivelevich and T. Szab\'o,
 Bart--Moe games, JumbleG and discrepancy,
 \emph{European J. Combin.} 28(4) (2007), 1131--1143.

\bibitem{Lu}
X. Lu,
 A matching game,
 \emph{Discrete Math.} 94 (1991), 199--207.

\bibitem{Matousek}
J. Matou\v{s}ek,
 \emph{Geometric Discrepancy: An Illustrated Guide},
 Springer, Berlin, 1999.

\bibitem{Roth1954}
K. F. Roth,
 On irregularities of distribution,
 \emph{Mathematika} 1 (1954), 73--79.

\bibitem{Spencer85}
J. Spencer,
 Six standard deviations suffice,
 \emph{Trans. Amer. Math. Soc.} 289(2) (1985), 679--706.

\end{thebibliography}
\end{document}